\documentclass[a4paper,12pt]{article}

\usepackage{amsmath}
\usepackage{amssymb}
\usepackage{amsthm}

\newtheorem{theorem}{Theorem}[section]
\newtheorem{lemma}[theorem]{Lemma}
\newtheorem{proposition}[theorem]{Proposition}

\theoremstyle{definition}
\newtheorem{example}[theorem]{Example}
\newtheorem{definition}[theorem]{Definition}
\newtheorem{remark}[theorem]{Remark}

\begin{document}

\title{Compactly supported Hamiltonian loops with a non-zero Calabi invariant }

\author{Asaf Kislev\footnote{Partially supported by European Research Council advanced grant 338809}}

\maketitle
\begin{abstract}
We give examples of compactly supported Hamiltonian loops with a non-zero Calabi invariant on certain open symplectic manifolds.

\end{abstract}

\section{Introduction}

Let $(M,\omega)$ be an open symplectic $2n$-dimensional manifold. Denote by $Ham(M)$ the group of Hamiltonian diffeomorphisms generated by compactly supported Hamiltonian functions.
Denote by $\widetilde{Ham}(M)$ the universal cover of $Ham(M)$. We write elements of $\widetilde{Ham}(M)$ as $[\{f_t\}_{t\in[0,1]}]$, where $\{f_t\}_{t\in[0,1]}$ is a smooth path of Hamiltonian diffeomorphisms with $f_0 = Id$, and $[\{f_t\}_{t\in[0,1]}]$ stands for the homotopy class of $\{f_t\}_{t\in[0,1]}$ with fixed end points. In what follows we use the notation $ \mathrm{Vol}(M) := \int_M \omega^n$.

Introduce the Calabi homomorphism $\mathrm{Cal}: \widetilde{Ham}(M) \to \mathbb{R}$, as \[\mathrm{Cal}([\{f_t\}_{t\in[0,1]}]) = \int_0^1 \int_M F_t \omega^n dt,\]
where $\{F_t\}_{t\in[0,1]}$ is the compactly supported Hamiltonian function whose flow is $\{f_t\}_{t\in[0,1]}$.

When $\omega$ is exact, the Calabi homomorphism vanishes on $\pi_1(Ham(M))$ and hence descends to $Ham(M)$ (see \cite[section 10.3]{1}). In this case \[\mathrm{Cal}([\{f_t\}_{t\in[0,1]}]) = \mathrm{Cal}([\{g_t\}_{t\in[0,1]}]),\] whenever $f_1 = g_1$.

Our goal is to give examples of open symplectic manifolds with a non exact symplectic form, for which $\mathrm{Cal}$ does not descend to $Ham(M)$. Equivalently, we wish to find a Hamiltonian loop generated by a compactly supported Hamiltonian function $\{H_t\}_{t \in [0,1]}$, such that \[\int_0^1\int_M H_t \omega^n dt \not=0.\]
Existence of such an example is indicated by McDuff in \cite[Remark 3.10]{2}. In addition, our interest to the problem was stimulated by \cite{6}.

An immediate corollary is that we get examples of open symplectic manifolds such that $\pi_1(Ham(M)) \not= 0$.

Let us present one geometric consequence of the non-vanishing of $\mathrm{Cal}$ on
\[\Pi:= \pi_1((Ham(M)) \subset \widetilde{Ham}(M)\;.\]
For an element $\gamma \in \Pi$ put
\[\ell(\gamma):= \inf_{\{f_t\}_{t \in [0,1]}} \int_0^1 (\max F_t - \min F_t) \; dt\;,\]
where the infimum is taken over all loops $\{f_t\}_{t \in [0,1]}$ representing $\gamma$ and $F$ stands for the compactly supported
Hamiltonian generating $\{f_t\}_{t \in [0,1]}$. Recall from \cite[Chapter 7]{5} that the set
\[\{\ell(\gamma)\;:\; \gamma \in \Pi \}\]
forms the {\it Hofer length spectrum} of ${Ham}(M)$. One readily checks that
\[\ell(\gamma) \geq |\mathrm{Cal}(\gamma)|\;\; \forall \gamma \in \Pi\;.\]
Therefore, when  $\mathrm{Cal}$ does not descend,
the Hofer length spectrum is non-trivial.

\vspace{4 mm}

In section 2 we state the main theorem which enables us to construct examples of Hamiltonian loops with a non-zero Calabi invariant. Its proof is given in section 3.

\section{Examples for Hamiltonian loops with a non-zero Calabi invariant}

\begin{definition}
Let $(X^{2n},\omega)$ be a closed symplectic manifold, and let $\lbrace f_t \rbrace _{t \in [0,1]}$ be a Hamiltonian $S^1$-action on $X$.
Let $z_0 \in X$ be a fixed point of the action. We say that $z_0$ is a \emph{Maslov-zero fixed point} if the loop \[\{d_{z_0}f_t\}_{t \in [0,1]} \subset Sp(T_{z_0}X,\omega),\] has Maslov index 0.

\end{definition}

\begin{theorem}
\label{t1}
Let $(X^{2n},\omega)$ be a closed symplectic manifold, and let $\lbrace f_t \rbrace _{t \in [0,1]}$ be a Hamiltonian $S^1$-action generated by a Hamiltonian function $F$. Let $z_0$ be a Maslov-zero fixed point of $\lbrace f_t \rbrace _{t \in [0,1]}$ which satisfies \[\int_{X} F \omega^n \neq \mathrm{Vol}(X) \cdot F(z_0).\]

Then there exists an open neighborhood $B$ that contains $z_0$ such that there exists a compactly supported Hamiltonian loop $\lbrace h_t \rbrace _{t \in [0,1]}$ in the open manifold $M =~ X \backslash \overline{B}$, which satisfies: \[\mathrm{Cal}([\{h_t\}_{t \in [0,1]}]) \neq 0,\]
and which also satisfies that it is homotopic to $\lbrace f_t \rbrace _{t \in [0,1]}$ in $Ham(X)$.

\end{theorem}

\begin{example}
\label{e1}

Let \[X = S^2(r_1) \times S^2(r_2) \subset \mathbb{R}^3 \times \mathbb{R}^3\] with $r_1 \neq r_2$, where $r_1,r_2$ stand for the radii. This is a closed symplectic manifold with the symplectic form $\omega := area_{S^2(r_1)} \oplus area_{S^2(r_2)}$.
Define the Hamiltonian function \[ F((\alpha_1,\alpha_2,\alpha_3),(\beta_1,\beta_2,\beta_3)) := 2\pi r_1 \alpha_3 + 2 \pi r_2 \beta_3.\]
The Hamiltonian flow of $F$ is an $S^1$-action.
The point \[s = ((0,0,r_1),(0,0,-r_2))\] is a fixed point of the flow of $F$.

 \textbf{Claim:} The fixed point $s$ is a Maslov-zero fixed point, and \[\int_{X} F \omega^n \neq \mathrm{Vol}(X) \cdot F(s).\]

 \textbf{Proof:}
 Near $s$ we can find an $S^1$-invariant neighbourhood which is $S^1$-equivariantly symplectomorphic to an open neighborhood of $0$ in $\mathbb{C}^2$ with the Hamiltonian $S^1$-action generated by:\footnote{More explicitly, if we look at $S^2 \times S^2$ with cylindrical coordinates $(\theta_1,\alpha_3,\theta_2,\beta_3)$ and on $\mathbb{C}^2$ with polar coordinates $(\sigma_1,\rho_1,\sigma_2,\rho_2)$, then we can write the symplectomorphism explicitly \[(\theta_1,\alpha_3,\theta_2,\beta_3) \mapsto (\theta_1,\sqrt{2r_1^2-2r_1\alpha_3},-\theta_2,\sqrt{2r_2\beta_3 + 2r_2^2}).\]} \[F' = \pi \cdot (-\left| z_1 \right| ^2 + \left| z_2\right| ^2),\]
where $(z_1,z_2) \in \mathbb{C}^2$. We see that the Hamiltonian flow in $\mathbb{C}^2$ is the loop of symplectic matrices $e^{-2\pi i t} \oplus e^{2\pi i t} \in Sp(4)$, which has Maslov index 0. We get that $s$ is a Maslov-zero fixed point.

We can calculate \[\int_{S^2(r_1)\times S^2(r_2)} \omega^2 \cdot F(s) = \mathrm{Vol}(S^2(r_1) \times S^2(r_2)) \cdot (2\pi r_1^2 - 2 \pi r_2^2) \neq 0,\]
while \[ \int_{S^2(r_1)\times S^2(r_2)} F \omega^2 = 0.\]

This proves the claim and we see that all the requirements of Theorem~ \ref{t1} are satisfied.

\begin{remark}
\label{r1}
We can extend Example \ref{e1} to higher dimensions.

Look at \[\left( S^2 \right)^m = S^2(r_1) \times S^2(r_1) \times \ldots \times S^2(r_1) \times S^2(r_2),\] with the Hamiltonian \[F = 2 \pi (r_1 h_1 + r_1 h_2 + \ldots + r_1 h_{m-1} + (m-1) \cdot r_2 h_m),\] where $h_i$ is the height function of the $i$-th copy of $S^2$ and $r_1 \neq r_2$. Define $s$ to be the fixed point \[s=((0,0,r_1),\ldots,(0,0,r_1),(0,0,-r_2)).\]
Now $s$ is a Maslov-zero fixed point, $F(s) \neq 0$ and $\int F \omega^m = 0$. We get again that all the requirements of Theorem \ref{t1} are satisfied.

\end{remark}

\end{example}

\vspace{4 mm}

\begin{example}
\label{e4}
Let $(X',\omega)$ be a closed symplectic manifold with an $S^1$-action which has a Maslov-zero fixed point $s$, and is generated by a non-constant Hamiltonian function $F$. Under these conditions we can always construct an open symplectic manifold which admits a compactly supported Hamiltonian loop with a non-zero Calabi invariant.

Normalize $F$ so that $F(s) = 0$. Take another fixed point $s'$ with $F(s') \neq 0$ (there are always at least two fixed points with different $F$-values because $F$ achieves maximum and minimum on $X'$). If $\int_{X'} F \omega^n \neq \mathrm{Vol}(X') \cdot F(s)$ then the requirements of Theorem \ref{t1} are satisfied right away.
Assume that $\int_{X'} F \omega^n = 0$. We first deal with the case where $F(s') > 0$. Choose a neighborhood $U$ of $s'$ so that $F \vert _U > 0$. Perform an equivariant symplectic blow-up at $s'$ (see \cite[section 6.1]{4}) with a small enough weight so that the symplectic ball $B$ that we cut out in the blow-up will be inside $U$. Call the resulting symplectic manifold $(X,\widetilde{\omega})$, and the resulting Hamiltonian function $\widetilde{F}$. The mean value of $\widetilde{F}$ will be less than zero because
\[ 0 = \int_{X'} F \omega^n = \int_{X} \widetilde{F} \widetilde{\omega}^n + \int_B F \omega^n .\]
Hence,
\[ \int_{X} \widetilde{F} \widetilde{\omega}^n = - \int_B F \omega^n .\]
Note that in a neighborhood of $s$, we have $\widetilde{F} = F$. We get that $s$ is a Maslov-zero fixed point of $\widetilde{F}$ and $\widetilde{F}(s)=0$.
Hence, the requirements of Theorem \ref{t1} are satisfied.

If $F(s') < 0$ then we can define $U$ to be a neighborhood of $s'$ so that $F \vert _U < 0$ and then continue as before.

Note that we can always perform an equivariant symplectic blow-up if the weight is small enough. From the equivariant Darboux theorem, we get that there is a small neighborhood of $s'$ which is equivariantly symplectomorphic to a neighborhood of zero in $\mathbb{C}^n$ with a linear symplectic $S^1$-action inside $U(n)$. Hence, we can choose a ball inside the neighborhood with radius $\lambda$ and get that the action naturally extends to the blow-up with weight $\lambda$.

\end{example}

\vspace{4 mm}

\begin{example}
\label{e3}

Here we give examples for open monotone symplectic manifolds which admits a Hamiltonian loop with a non-zero Calabi invariant.

Consider the $S^1$-action defined on $\mathbb{CP}^n$ by
\[f_t([x_0:x_1:\ldots:x_n]) = [x_0:e^{2\pi ia_1t}x_1:\ldots:e^{2 \pi i a_n t}x_n],\]
for some $(a_1,\ldots,a_n) \in \mathbb{Z}^n$ where $\sum_{i=1}^n a_i = 0$ and $(a_1,\ldots,a_n) \neq 0$. Denote the moment map by $F$.
The point $s = [1:0:\ldots:0]$ is a Maslov-zero fixed point of $F$. Perform a monotone $\mathbb{T}^n$-equivariant symplectic blow-up at another fixed point where the value of the moment map is different. Call the resulting manifold $X$. We get that
\[\int_{X} F \omega^n \neq \mathrm{Vol}(X) \cdot F(s).\]
Hence, all the requirements of Theorem \ref{t1} are satisfied so we can construct an open monotone symplectic manifold $M := X \backslash \overline{B}$ which admits a Hamiltonian loop with a non-zero Calabi invariant. Note that $M$ remains monotone.

In Remark \ref{r3} below we shall use the monotonicity in order to show that the non-contractible Hamiltonian loop that we have constructed in $Ham(M)$ is also non-contractible in $Ham(X)$.

\end{example}

\vspace{4 mm}

\begin{remark}
\label{r3}
In the previous examples we considered an open manifold $M$ which is a subset of a bigger closed manifold $X$.
One could ask if the non-contractible loops that we have constructed in $Ham(M)$ remain non-contractible in $Ham(X)$.
The constructed loop in $Ham(M)$ is always homotopic to the original $S^1$-action in $Ham(X)$ so we get that it remains to check wether the original $S^1$-action is non-contractible in $Ham(X)$.

The answer to this question is, in general, no, as can be seen in the following example:
Let $X = S^2(r_1) \times S^2(r_2)$ with $r_1 \neq r_2$. Consider the Hamiltonian function
\[ F((\alpha_1,\alpha_2,\alpha_3),(\beta_1,\beta_2,\beta_3)) := 4\pi r_1 \alpha_3 + 4 \pi r_2 \beta_3.\]
The Hamiltonian flow of $F$ is the 2-turn rotation around the $\alpha_3$-axis times the 2-turn rotation around the $\beta_3$-axis.
It is a known fact that
\[ \pi_1(Ham(S^2)) = \mathbb{Z}_2, \]
and that the generator is the 1-turn rotation (see \cite[section 7.2]{5}).
We get that the 2-turn rotation is a contractible loop in $Ham(S^2)$ and hence the flow of $F$ is contractible in $Ham(X)$.

The point $s = ((0,0,r_1),(0,0,-r_2))$ is a Maslov-zero fixed point and \[\int_{X} F \omega^n \neq \mathrm{Vol}(X) \cdot F(s).\]
The requirements of Theorem \ref{t1} are satisfied so we can construct the open manifold $M$ and the Hamiltonian loop $\{h_t\}_{t \in [0,1]}$, such that $\{h_t\}_{t \in [0,1]}$ is non-contractible in $Ham(M)$. The loops $\{h_t\}_{t \in [0,1]}$ and $\{f_t\}_{t \in [0,1]}$ are in the same homotopy class in $\pi_1(Ham(X))$ so we get that $\{h_t\}_{t \in [0,1]}$ is contractible in $Ham(X)$ and non-contractible in $Ham(M)$.

We will now show that the situation is different for monotone manifolds. For a closed monotone symplectic manifold $X$ the Hamiltonian loops that we are considering will always be non-contractible in $Ham(X)$.

\begin{proposition}
Let $X$ be a closed monotone symplectic manifold which admits a Hamiltonian $S^1$-action $\lbrace f_t \rbrace _{t \in [0,1]}$ with a Maslov-zero fixed point $s$ such that
\[\int_{X} F \omega^n \neq \mathrm{Vol}(X) \cdot F(s),\]
where $F$ is the moment map.
Then the loop $\lbrace f_t \rbrace _{t \in [0,1]}$ is non-contractible in $Ham(X)$.
\end{proposition}
\vspace{4 mm}

\begin{proof}
Consider the mixed action-Maslov invariant of the loop $\lbrace f_t \rbrace _{t \in [0,1]}$ (see \cite{8}). We calculate it via the constant loop $\lbrace f_t(s) \rbrace _{t \in [0,1]} \subset X$ with the constant spanning disk $\overline{s}:D^2 \to X$. Normalize $F$ so that its mean value will be zero, and calculate the symplectic action
\[A(\lbrace f_t \rbrace _{t \in [0,1]} ,\overline{s}) = \int_{D^2} \overline{s}^*\omega - \int_0^1F(s) dt = -F(s) \neq 0.\]
The Maslov index of the loop $\{d_{s}f_t\}_{t \in [0,1]}$ is zero so we get that the mixed action-Maslov invariant is
\[ I = A(\lbrace f_t \rbrace _{t \in [0,1]} ,\overline{s}) \neq 0. \]
Hence the loop $\lbrace f_t \rbrace _{t \in [0,1]}$ is non-contractible in $Ham(X)$.
\end{proof}
\end{remark}

\vspace{4 mm}

\section{Proof of Theorem \ref{t1}}
In order to prove Theorem \ref{t1} we will need the following lemma.

\begin{lemma}
\label{l1}
Let $\lbrace A_t \rbrace _{t \in [0,1]} \subset Sp(2n)$ be a contractible loop of symplectic matrices.
Then for any $\epsilon > 0$ and any open ball $B$ around zero in the standard symplectic $\mathbb{R}^{2n}$, there is a smaller ball $B'$ around zero and a Hamiltonian loop $\{g_t\}_{t \in [0,1]} \subset Ham(\mathbb{R}^{2n})$ such that:
\begin{enumerate}
\item
$g_t$ is supported in $B$.
\item
$g_t \vert _{B'} = A_t \vert _{B'}$.
\item
$\mid \mathrm{Cal}([\{g_t\}_{t\in[0,1]}]) \mid < \epsilon $.
\item
$\{g_t\}_{t \in [0,1]}$ is a contractible loop in $Ham(\mathbb{R}^{2n})$.
\end{enumerate}

\end{lemma}
\vspace{5 mm}
\begin{proof}

Let $A_{s,t}$ be a smooth homotopy of the loop $\lbrace A_t \rbrace _{t\in [0,1]}$ to the identity, such that $A_{0,t} = Id, A_{1,t} = A_t$, and $A_{s,0}=A_{s,1} = Id$.

\vspace{4 mm}

\textbf{Claim:} There is a function $H: \mathbb{R}^{2n} \times I \times I \rightarrow \mathbb{R}$ such that for a fixed $t$, $\{H(\cdot ,s,t)\}_{s\in[0,1]}$ generates $\{A_{s,t}\}_{s\in[0,1]}$ as the Hamiltonian flow with respect to the time variable $s$, and if we put $H_{s,t} := H(\cdot ,s,t)$ then:
\begin{enumerate}
\item
$H$ is smooth.
\item
$H_{s,t}(0) = 0$.
\item
$H_{s,0} = H_{s,1} = 0$.
\end{enumerate}

\textbf{Proof:} A path of symplectic matrices $\{B_t\}_{t\in[0,1]}$ is generated as Hamiltonian flow by a Hamiltonian function of the form $<x,Q_tx>$, where $Q_t = \frac{1}{2} J \frac{\partial B_t}{\partial t} B_t^{-1}$ and $J$ is the standard linear complex structure on $\mathbb{R}^{2n}$. In our situation $H_{s,t} = <x,Q_{s,t}x>$ is a Hamiltonian function that generates the flow $s \mapsto A_{s,t}$ for $Q_{s,t} = \frac{1}{2} J \frac{\partial A_{s,t}}{\partial s} A_{s,t}^{-1}$.

We get that $H$ is smooth, \[H_{s,t}(0) = <0,Q_{s,t}0>=0,\] and \[H_{s,0}=H_{s,1}=0.\] The latter assertion holds because \[A_{s,0}=A_{s,1}=Id,\] so \[\frac{\partial A_{s,0}}{\partial s} = \frac {\partial A_{s,1}} {\partial s} = 0.\]
This proves the claim.

\vspace{4 mm}

Choose three balls $B_1 \subset B_2 \subset B_3 \subset B$ such that for each $t\in [0,1]$, $A_tB_1 \subset B_2$. Choose a cut-off function $a:\mathbb{R}^{2n} \to \mathbb{R}$ such that $a \vert_{B_2} = 1$ and $a\vert_{\mathbb{R}^{2n}\backslash B_3} = 0$.

 For a fixed $t$, define $\{g_{s,t}\}_{s\in[0,1]}$ to be the Hamiltonian flow generated by the Hamiltonian function $\lbrace a\cdot H_{s,t} \rbrace _{s \in [0,1]}$ with the time variable $s$.

 For a fixed $s$, $\lbrace g_{s,t} \rbrace _{t \in [0,1]}$ is a loop with respect to $t$. That is true because $H_{s,0} = H_{s,1} = 0$ for every $s$, so $g_{s,0} = g_{s,1} = Id$.

  Define $\lbrace G_{s,t}\rbrace _{t \in [0,1]}$ to be the Hamiltonian function that generates $g_{s,t}$ for a fixed $s$ with the time variable $t$, normalized so that $G_{s,t}(0) = 0$. Denote $g_t := g_{1,t}$. Note that \[g_t \vert_{B_1} = A_t \vert _{B_1}\] for each $t$, because $a \vert _{B_2} =1$ and $A_tB_1 \subset B_2$. Note also that $g_{0,t} = Id$ because it is the time-0 map of the flow with the time variable $s$, so we get that \[G_{0,t} = 0\] for each $t \in [0,1]$.

\vspace{4 mm}

\textbf{Claim:} For each $\epsilon > 0$ there is $\delta > 0$ such that if $B_3$ is with radius less than $\delta$, then \[\left| \int_0^1 \int_{\mathbb{R}^{2n}} G_{1,t} \omega_0^n \; dt \right| < \epsilon,\] where $\omega_0$ is the standard symplectic form on $\mathbb{R}^{2n}$.

\vspace{4 mm}

\textbf{Proof:} The Hamiltonian functions $G_{s,t}$ generates $g_{s,t}$ with time $t$, and $a\cdot H_{s,t}$ generates $g_{s,t}$ with time $s$. We can use a known formula (see  \cite[section 6.1]{5}), and get that \[\frac {\partial G_{s,t}}{\partial s} = a \cdot \frac {\partial H_{s,t}} {\partial t} - \lbrace G_{s,t},a \cdot H_{s,t} \rbrace. \]
Note that $G_{0,t} = 0$, so if we integrate over $s$ we will get \[ G_{1,t} = \int_0^1 a \cdot \frac {\partial H_{s,t}} {\partial t} - \lbrace G_{s,t},a \cdot H_{s,t} \rbrace ds. \]
Note that $G_{1,t} = 0$ outside $B_3$.
Multiply with $\omega_0^n$ and integrate over $\mathbb{R}^{2n}$ to get
\[\int_{\mathbb{R}^{2n}} G_{1,t} \omega_0^n = \int_{B_3} a \cdot \int_0^1 \frac {\partial H_{s,t}} {\partial t} ds \; \omega_0^n - \int_0^1 ds \int_{B_3} \lbrace G_{s,t},a \cdot H_{s,t} \rbrace \omega_0^n. \]
For any two smooth functions $F_1,F_2$, the form $\lbrace F_1,F_2 \rbrace \omega_0^n$ is exact, so from this we get that \[\int_0^1 ds \int_{B_3} \lbrace G_{s,t},a \cdot H_{s,t} \rbrace \omega_0^n = 0.\]

Now we get that \[\left| \int_{\mathbb{R}^{2n}} G_{1,t}\omega_0^n \right| = \left| \int_{B_3} a \cdot \int_0^1 \frac {\partial H_{s,t}} {\partial t} ds \; \omega_0^n \right| \leq \int_{B_3} \int_0^1 \left| \frac {\partial H_{s,t}} {\partial t} \right| ds \; \omega_0^n.\]
 We know that $\frac {\partial H_{s,t}} {\partial t}$ is a smooth function so it is bounded on $B \times I \times I$. Hence for every $t,s \in [0,1]$ and for every $\delta$ such that $B_3 \subset B$, we have that $\left| \frac {\partial H_{s,t}} {\partial t} \right| < K$ for some $K > 0$. We get that \[\int_{B_3} \int_0^1 \left| \frac {\partial H_{s,t}} {\partial t} \right| ds \; \omega_0^n  \leq K \cdot \int_{B_3} \omega_0^n.\]
This means that we can choose $\delta$ such that \[\int_{B_3} \int_0^1 \left| \frac {\partial H_{s,t}} {\partial t} \right| ds \; \omega_0^n < \epsilon.\]
If we integrate over $t$ we will get that \[\left| \int_0^1 \int_{\mathbb{R}^{2n}} G_{1,t} dt \right| \leq \int_0^1 \int_{B_3} \int_0^1 \left| \frac {\partial H_{s,t}} {\partial t} \right| ds \; \omega_0^n \; dt < \epsilon.\]
This proves the claim.

\vspace{4 mm}

We get that $\{g_t\}_{t\in[0,1]}$ is supported in $B_3 \subset B$ , $g_t \vert _{B_1} = A_t \vert _{B_1}$ and for any $\epsilon > 0$ we can choose the radius of $B_3$ so that $\left| \mathrm{Cal}([\{g_t\}_{t\in[0,1]}]) \right| < \epsilon$, so we can define $B' := B_3$. Note also that the loop $\{g_t\}_{t\in[0,1]}$ is homotopic to the identity with the homotopy $\{g_{s,t}\}_{s \in [0,1]}$.

\vspace{4 mm}

\noindent This proves the lemma.
\end{proof}

\vspace{4 mm}

\begin{proof}[Proof of Theorem \ref{t1}]
Assume that $B_2$ is an equivariant Darboux ball around $z_0$ (see \cite{3}, section 3.1). We choose the Darboux coordinates on $B_2$ such that $z_0$ is identified with $0$ and the path $f_t \vert _{B_2}$ is identified with the path $A_t \in Symp(2n)$. From the fact that $z_0$ is a Maslov-zero fixed point, we know that the loop $\{A_t\}_{t\in[0,1]}$ has Maslov index 0, and hence it is a contractible loop in $Sp(2n)$.
Normalize $F$ so that $F(z_0) = 0$ and $\int_X F \omega^n \neq 0$.
Set \[\epsilon = \frac {\left| \int_X F \omega^n \right|} {2}.\]

Use Lemma \ref{l1} to define the contractible loop $\lbrace g_t \rbrace _{t \in [0,1]}$ and the ball $B_1$ such that for each $t \in [0,1]$, $g_t$ is supported in $B_2$, $g_t \vert _{B_1} = A_t \vert _{B_1}$, and \[\mid \mathrm{Cal}([\{g_t\}_{t\in[0,1]}]) \mid < \epsilon .\]
Define \[h_t := g_t^{-1} \circ f_t.\]
From the fact that $\lbrace g_t \rbrace _{t \in [0,1]}$ is a contractible loop we know that $\lbrace h_t \rbrace _{t \in [0,1]}$ is homotopic to $\lbrace f_t \rbrace _{t \in [0,1]}$ in $Ham(X)$.
Note also that \[h_t \vert _{B_1} = Id. \]
Define $H_t$ as the Hamiltonian function generating $h_t$: \[H_t := (F-G_t) \circ g_t. \]
Note that $G_t \vert _{B_1} = F \vert _{B_1}$.
Hence we get that $H_t \vert _{B_1} = 0$. Choose a ball $\overline{B} \subset B_1$. The flow $h_t$ is defined on $X \backslash \overline{B}$ and $H_t \vert _{B_1 \backslash \overline{B}} = 0$ so we get that $H_t$ is compactly supported on $X \backslash \overline{B}$.
From the definition of the loop $\lbrace g_t \rbrace _{t \in [0,1]}$ we have that
\[\left|\int_0^1 \int_X H_t \omega^n \; dt - \int_X F \omega^n \right| = \left| \int_0^1\int_X G_{t}\omega^n \; dt \right| = |\mathrm{Cal}([\{g_t\}_{t\in[0,1]}])| < \epsilon = \frac {\left| \int_X F \omega^n \right|} {2}. \]
From this and from the fact that $\int_X F \omega^n \neq 0$ we get that $\int_0^1\int_X H_t \omega^n \; dt \neq 0$. However, \[\int_0^1 \int_X H_t \omega^n \; dt = \int_0^1\int_{X\backslash \overline{B}} H_t \omega^n \; dt = \mathrm{Cal}([\{h_t\}_{t\in[0,1]}]).\] This completes the proof.

\end{proof}

\vspace{4 mm}

\textbf{Acknowledgement:}  This article was written under the guidance of Professor Leonid Polterovich. I also wish to thank Strom Borman, Sobhan Seyfaddini and Daniel Rosen for helpful discussions and suggestions. I am grateful to Yael Karshon for pointing out an inaccuracy in the first draft of the article. I wish to thank the anonymous referee for his/her thorough review and very useful comments and specifically for his/her suggestion to use the mixed action-Maslov homomorphism in order to show that the loops we are considering are non-contractible on monotone manifolds.

Asaf Kislev

School of Mathematical Sciences

Tel Aviv University

Tel Aviv 6997801, Israel

asafkisl@post.tau.ac.il

\end{document}